\documentclass[12pt]{amsart}
\usepackage{amsmath,amssymb,amsthm,amsfonts,appendix, youngtab,graphicx}
\usepackage[curve]{xy}
\usepackage{tikz}
\newtheorem{theorem}{Theorem}[section]
\theoremstyle{definition}
\newtheorem{definition}[theorem]{Definition}
\newtheorem{lemma}[theorem]{Lemma}
\newtheorem{proposition}[theorem]{Proposition}
\newtheorem{corollary}[theorem]{Corollary}
\newtheorem{example}[theorem]{Example}

\theoremstyle{remark}

\numberwithin{equation}{section}
\def\vs{\vspace{0.2cm}}

\begin{document}

\title {Symmetry in maximal ${\bf (s-1,s+1)}$ cores}
\author{Rishi Nath}
\address{York College/City University of New York}
\email{rnath@york.cuny.edu}

\subjclass[2000]{203j0}

\date{}
\dedicatory{}
\begin{abstract}
We explain a ``curious symmetry" for maximal $(s-1,s+1)$-core partitions first observed by by T. Amdeberhan and E. Leven. Specifically, using the $s$-abacus, we show such partitions have empty $s$-core and that their $s$-quotient is comprised of 2-cores. This imposes strong conditions on the partition structure, and implies both the Amdeberhan-Leven result and additional symmetry. We also find a more general family of partitions that exhibits these symmetries.
\end{abstract}
\maketitle
\section{Introduction}
The study of simultaneous core partitions, which began only fifteen years ago, has seen a recent spike of interest.  Much of the attention has focused around either a conjecture of Armstrong on the average size of an $(s,t)$-core or generalizing known results on the $(s,s+1)$ (Catalan) case.  Results in a recent paper of Amdeberhan and Leven deviate from this slightly to examine $(s-1,s+1)$-cores in the case where $s$ is even and greater than 2; they note a symmetry in the set of first column hook numbers of $\kappa_{s\pm1},$ the $(s-1,s+1)$-core of maximal size. [Theorem \ref{AmLev} in this paper states their result.]

\vs
Hidden by their proof (which involves involves the integral and fractional parts of a real number) is a connection with the $s$-core and $s$-quotient structure viewed on the $s$-abacus. From this vantage point, the Amdeberhan-Leven theorem is a result on the symmetry of runners (columns) of the $s$-abacus of maximal $(s-1,s+1)$-cores.

Given a partition $\lambda,$ let $\lambda^0$ be the $s$-core of $\lambda$ and $(\lambda_{(0)},\lambda_{(1)}, \cdots, \lambda_{(s-1)})$ be the $s$-quotient of $\lambda.$ Let $\kappa_{s\pm1}$ be the unique maximal simultaneous $(s-1,s+1)$-core partition and $\tau_{\ell}=(\ell,\ell-1,\ell-2,\cdots,1)$ be the $\ell$-th 2-core partition. We state our main theorem.
\begin{theorem} \label{piquo} Let $s=2k>2$. Then $\kappa_{s-1,s+1}$ has the following $s$-core and $s$-quotient structure:
\begin{enumerate}
\item $(\kappa_{s-1,s+1})^0=\emptyset.$
\item ${\kappa_{s\pm1}}_{(i)}={\kappa_{s\pm1}}_{(s-i-1)}=\tau_{k-i-1}$ where $0<i<k-1$.
\end{enumerate}
\end{theorem}

\begin{figure}
\caption{The 8-abacus of $\kappa_{7,9}$}
\hspace{-1cm}
\begin{center}
\scalebox{0.3}{\begin{tikzpicture}
\draw (0cm,-8cm) node [circle, minimum size=1.5cm, inner sep=0pt, draw=none, anchor=south west] {\Huge $0$};
\draw (2cm,-8cm) node [circle, minimum size=1.5cm, inner sep=0pt, draw, anchor=south west] {\Huge $1$};
\draw (4cm,-8cm) node [circle, minimum size=1.5cm, inner sep=0pt, draw, anchor=south west] {\Huge $2$};
\draw (6cm,-8cm) node [circle, minimum size=1.5cm, inner sep=0pt, draw, anchor=south west] {\Huge $3$};
\draw (8cm,-8cm) node [circle, minimum size=1.5cm, inner sep=0pt, draw, anchor=south west] {\Huge $4$};
\draw (10cm,-8cm) node [circle, minimum size=1.5cm, inner sep=0pt, draw, anchor=south west] {\Huge $5$};
\draw (12cm,-8cm) node [circle, minimum size=1.5cm, inner sep=0pt, draw, anchor=south west] {\Huge $6$};
\draw (14cm,-8cm) node [circle, minimum size=1.5cm, inner sep=0pt, draw=none, anchor=south west] {\Huge $7$};
\draw (0cm,-6cm) node [circle, minimum size=1.5cm, inner sep=0pt, draw, anchor=south west] {\Huge $8$};
\draw (2cm,-6cm) node [circle, minimum size=1.5cm, inner sep=0pt, draw=none, anchor=south west] {\Huge $9$};
\draw (4cm,-6cm) node [circle, minimum size=1.5cm, inner sep=0pt, draw, anchor=south west] {\Huge $10$};
\draw (6cm,-6cm) node [circle, minimum size=1.5cm, inner sep=0pt, draw, anchor=south west] {\Huge $11$};
\draw (8cm,-6cm) node [circle, minimum size=1.5cm, inner sep=0pt, draw, anchor=south west] {\Huge $12$};
\draw (10cm,-6cm) node [circle, minimum size=1.5cm, inner sep=0pt, draw, anchor=south west] {\Huge $13$};
\draw (12cm,-6cm) node [circle, minimum size=1.5cm, inner sep=0pt, draw=none, anchor=south west] {\Huge $14$};
\draw (14cm,-6cm) node [circle, minimum size=1.5cm, inner sep=0pt, draw, anchor=south west] {\Huge $15$};

\draw (0cm,-4cm) node [circle, minimum size=1.5cm, inner sep=0pt, draw=none, anchor=south west] {\Huge $16$};
\draw (2cm,-4cm) node [circle, minimum size=1.5cm, inner sep=0pt, draw, anchor=south west] {\Huge $17$};
\draw (4cm,-4cm) node [circle, minimum size=1.5cm, inner sep=0pt, draw=none, anchor=south west] {\Huge $18$};
\draw (6cm,-4cm) node [circle, minimum size=1.5cm, inner sep=0pt, draw, anchor=south west] {\Huge $19$};
\draw (8cm,-4cm) node [circle, minimum size=1.5cm, inner sep=0pt, draw, anchor=south west] {\Huge $20$};
\draw (10cm,-4cm) node [circle, minimum size=1.5cm, inner sep=0pt, draw=none, anchor=south west] {\Huge $21$};
\draw (12cm,-4cm) node [circle, minimum size=1.5cm, inner sep=0pt, draw, anchor=south west] {\Huge $22$};
\draw (14cm,-4cm) node [circle, minimum size=1.5cm, inner sep=0pt, draw=none, anchor=south west] {\Huge $23$};
\draw (0cm,-2cm) node [circle, minimum size=1.5cm, inner sep=0pt, draw, anchor=south west] {\Huge $24$};
\draw (2cm,-2cm) node [circle, minimum size=1.5cm, inner sep=0pt, draw=none, anchor=south west] {\Huge $25$};
\draw (4cm,-2cm) node [circle, minimum size=1.5cm, inner sep=0pt, draw, anchor=south west] {\Huge $26$};
\draw (6cm,-2cm) node [circle, minimum size=1.5cm, inner sep=0pt, draw=none, anchor=south west] {\Huge $27$};
\draw (8cm,-2cm) node [circle, minimum size=1.5cm, inner sep=0pt, draw=none, anchor=south west] {\Huge $28$};
\draw (10cm,-2cm) node [circle, minimum size=1.5cm, inner sep=0pt, draw, anchor=south west] {\Huge $29$};
\draw (12cm,-2cm) node [circle, minimum size=1.5cm, inner sep=0pt, draw=none, anchor=south west] {\Huge $30$};
\draw (14cm,-2cm) node [circle, minimum size=1.5cm, inner sep=0pt, draw, anchor=south west] {\Huge $31$};

\draw (0cm,0cm) node [circle, minimum size=1.5cm, inner sep=0pt, draw=none, anchor=south west] {\Huge $32$};
\draw (2cm,0cm) node [circle, minimum size=1.5cm, inner sep=0pt, draw, anchor=south west] {\Huge $33$};
\draw (4cm,0cm) node [circle, minimum size=1.5cm, inner sep=0pt, draw=none, anchor=south west] {\Huge $34$};
\draw (6cm,0cm) node [circle, minimum size=1.5cm, inner sep=0pt, draw=none, anchor=south west] {\Huge $35$};
\draw (8cm,0cm) node [circle, minimum size=1.5cm, inner sep=0pt, draw=none, anchor=south west] {\Huge $36$};
\draw (10cm,0cm) node [circle, minimum size=1.5cm, inner sep=0pt, draw=none, anchor=south west] {\Huge $37$};
\draw (12cm,0cm) node [circle, minimum size=1.5cm, inner sep=0pt, draw, anchor=south west] {\Huge $38$};
\draw (14cm,0cm) node [circle, minimum size=1.5cm, inner sep=0pt, draw=none, anchor=south west] {\Huge $39$};
\draw (0cm,2cm) node [circle, minimum size=1.5cm, inner sep=0pt, draw, anchor=south west] {\Huge $40$};
\draw (2cm,2cm) node [circle, minimum size=1.5cm, inner sep=0pt, draw=none, anchor=south west] {\Huge $41$};
\draw (4cm,2cm) node [circle, minimum size=1.5cm, inner sep=0pt, draw=none, anchor=south west] {\Huge $42$};
\draw (6cm,2cm) node [circle, minimum size=1.5cm, inner sep=0pt, draw=none, anchor=south west] {\Huge $43$};
\draw (8cm,2cm) node [circle, minimum size=1.5cm, inner sep=0pt, draw=none, anchor=south west] {\Huge $44$};
\draw (10cm,2cm) node [circle, minimum size=1.5cm, inner sep=0pt, draw=none, anchor=south west] {\Huge $45$};
\draw (12cm,2cm) node [circle, minimum size=1.5cm, inner sep=0pt, draw=none, anchor=south west] {\Huge $46$};
\draw (14cm,2cm) node [circle, minimum size=1.5cm, inner sep=0pt, draw, anchor=south west] {\Huge $47$};
\end{tikzpicture}}
\end{center}
\end{figure}
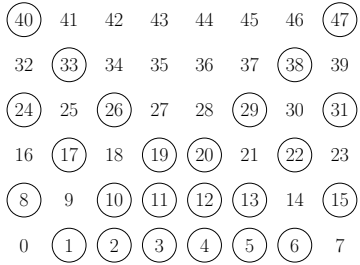
\begin{center}
\begin{figure}
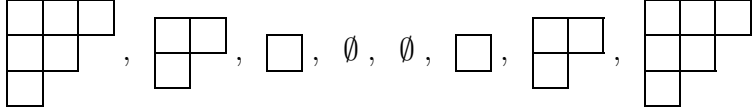

\hspace{-1cm}
\caption{8-quotient of $\kappa_{7,9}$}
\Yvcentermath1
$\yng(3,2,1)\;, \;\; \yng(2,1)\;,\;\; \yng(1)\;,\;\; \emptyset\;,\;\; \emptyset\;,\;\; \yng(1)\;,\; \;\yng(2,1)\;,\; \;\yng(3,2,1)$
\end{figure}
\end{center}
In Section 3.1 we describe the $s$-abacus of $\kappa_{s\pm1}$, which we use to prove Theorem \ref{piquo}. We provide an alternate proof of the Amdeberhan-Leven result in Section 3.2. In Section 4.1 we demonstrate an additional symmetry in the rows of the $s$-abacus of $\kappa_{s-1,s+1}$. We formalize both the runner and row symmetries exhibited by $\kappa_{s\pm1}$ in Section 4.2, and describe the most general family of partitions which satisfy them. 
\begin{example} The 8-abacus of $\kappa_{7,9}$ and the associated 8-quotient are shown in {\bf Figure 1} and {\bf Figure 2} respectively. [Note: the 8-quotient consists of a sequence of 2-core partitions, arising from the structure of the 8-abacus.]
\end{example}
\section{preliminaries}
\subsection{Basic definitions}
Let $\mathbb{N}=\{0,1,\cdots\}$ and
$n\in\mathbb{N}$. A {\it partition} $\lambda$ of
$n$ is defined as a finite, non increasing sequence of positive integers $(\lambda_1,\lambda_2,\cdots)$ that sums to $n$.  Each $\lambda_{\gamma}$ is known as a {\it component} of $\lambda$. Then $\sum_{\gamma}\lambda_{\gamma}=n,$ and $\lambda$ is said to have {\it size} $n$, denoted $|\lambda|=n.$ We also use the notation $\lambda^m_i$ to indicate that $\lambda_i$ occurs $m$-times as a component of $\lambda$.

\vs
The {\it Young diagram} $[\lambda]$ is a graphic representation of $\lambda$ in which rows of boxes corresponding to the integer values in the partition sequence are left-aligned. Then $\lambda^*$ is the {\it conjugate partition} of $\lambda$  obtained by exchanging rows and columns of the Young diagram of $\lambda.$ Then $\lambda$ is {\it self-conjugate} if $\lambda=\lambda^*.$ Using matrix notation, a {\it hook} $h_{\iota \gamma}$ of $[\lambda]$ with {\it corner} $(\iota, \gamma)$ is the set of boxes to the right of $(\iota, \gamma)$ in the same row, below $(\iota, \gamma)$ in the same column, and $(\iota, \gamma)$ itself. Given $h_{\iota \gamma},$ its {\it length} $|h_{\iota \gamma}|$ is the number of boxes in the hook. The set $\{h_{1\gamma}\}$ are the {\it first-column hooks} of $\lambda$.

\vs
One can {\it remove} a hook $h$ of $\lambda$ by deleting boxes in $[\lambda]$ which comprise $h$ and migrating any remaining detached boxes up-and-to-the-left. In this way a new partition $\lambda'$ of size $n-|h_{\iota \gamma}|$ is obtained. An {\it s-hook} is a hook of length $s$. An {\it s-core partition} $\lambda$ is one in which no hook of length $s$ appears in the Young diagram. 

\subsection{Simultaneous $(s,t)$-core partitions}
Let $r,s,t$ be positive integers. A {\it simultaneous (s,t)-core partition} is one in which no hook of length $s$ or $t$ appears. In 1999, J. Anderson \cite{A} proved when $(s,t)=1,$ there are exactly $\binom{s+t}{t}/(s+t)$ simultaneous $(s,t)$-cores. Subsequent work by B. Kane \cite{Ka}, J. Olsson and D. Stanton \cite{O-S}, J. Vandehey \cite{V} confirmed the existence of a unique {\it maximal} $(s,t)$-core of size $\frac{(s^2-1)(t^2-1)}{24}$ which contains all other $(s,t)$-cores. This maximal simultaneous $(s,t)$-core partition is denoted by $\kappa_{s,t}.$ [A. Tripathi \cite{T} and M. Fayers \cite{F} obtained some of the results above using different methods.] When it is convenient we will denote $\kappa_{s-1,s+1}$ by $\kappa_{s\pm1}$.
\begin{theorem}\label{s2t2}[Olsson-Stanton, Theorem 4.1, \cite{O-S}] Suppose $(s,t)=1$. There is a unique maximal simultaneous $(s,t)$-core $\kappa_{s,t}$ of size $\frac{(s^2-1)(t^2-1)}{24}.$ In particular, $\kappa_{s,t}$ is self-conjugate.
\end{theorem}
\begin{figure}
\begin{center}
\raisebox{.05in}{\scalebox{0.35}{\begin{tikzpicture}
\draw (0,-13.5cm) node [rectangle, minimum size=1.5cm, inner sep=0pt, draw, anchor=south west]  {\Huge $7$};
\draw (1.5cm,-13.5cm) node [rectangle, minimum size=1.5cm, inner sep=0pt, draw, anchor=south west]  {\Huge $4$};
\draw (3.0cm,-13.5cm) node [rectangle, minimum size=1.5cm, inner sep=0pt, draw, anchor=south west]  {\Huge $2$};
\draw (4.5cm,-13.5cm) node [rectangle, minimum size=1.5cm, inner sep=0pt, draw, anchor=south west] {\Huge $1$};
\draw (0,-15.0cm) node [rectangle, minimum size=1.5cm, inner sep=0pt, draw, anchor=south west]  {\Huge $4$};
\draw (1.5cm,-15.0cm) node [rectangle, minimum size=1.5cm, inner sep=0pt, draw, anchor=south west]  {\Huge $1$};
%
\draw (0,-16.5cm) node [rectangle, minimum size=1.5cm, inner sep=0pt, draw, anchor=south west]  {\Huge $2$};
\draw (0,-18.0cm) node [rectangle, minimum size=1.5cm, inner sep=0pt, draw, anchor=south west]  {\Huge $1$};
\end{tikzpicture}}}
\end{center}
\caption{Young diagram (with hook lengths) of $\kappa_{3,5}$}
\end{figure}

\vs
A recent paper of D. Armstrong, C. Hanusa and B. Jones \cite{A-H-J} includes a conjecture (the Armstrong conjecture) that the average size of a $(s,t)$-core is $\frac{(s+t+1)(s-1)(t-1)}{24}$. R. Stanley and F. Zenello \cite{S-Z} subsequently resolved the Catalan ($t=s+1$) case of the Armstrong conjecture; they employ a bijection between lower ideals in the poset $P_{s,t}$ and simultaneous $(s,t)$-cores. [Here $P_{s,t}$ is the partially ordered set whose elements are all positive integers not contained in the numerical semigroup generated by ${s, t}.$ The partial order requires $z_1 \in P_{s,t}$ to cover $z_2 \in P_{s,t}$ if $z_1-z_2$ is either $s$ or $t$.] Under this map a lower ideal $I$ of $P_{s,t}$ corresponds to an $(s,t)$-core partition whose first-column hook lengths are exactly the values in $I$. Then $P_{s,t}$ corresponds to $\kappa_{s,t}.$ 

\vs
These two papers have led to renewed interest in simultaneous core partitions. The Armstrong conjecture has been verified for self-conjugate partitions by W. Chen, H. Huang, and L. Wang \cite{C-H-W} and for $(s, ms+1)$ by A. Aggarwal \cite{AA1}. T. Amdeberhan and E. Leven \cite{A-L} extended Stanley and Zanello's bijection to lower poset ideals and simultaneous $(s_1,s_2,\cdots,s_k)$-cores. Several conjectures of T. Ambederhan \cite{A} on the maximal and average size simultaneous $(s,s+1,s+2)$-cores have been proved first by J. Yang, M. Zhong and R. Zhou \cite{Y-Z-Z} and later by H. Xiong \cite{X}. A. Aggarwal has also proved a partial converse to a theorem of Vandehey on the containment of simultaneous $(r,s,t)$-cores \cite{AA2}. 
\subsection{A ``curious symmetry"}
Amdeberhan and Leven also examine $P_{r,r+2}$ for $r$ odd. They first construct a $(r-1)\times(r+1)$ rectangle $R$ as follows: the bottom-left corner is labelled by 1, the numbers increase from left-to-right and bottom-to-top, and the largest position, in the upper-right corner, is labeled by $(r-1)(r+1)$. If $x\in P_{r,r+2}$ then $x$ is entered into this rectangle, otherwise the position is left blank. Using a {\it runner-row index}, counting runners (or columns) $a$ from left-to-right in the $x$-coordinate ($1\leq a\leq r+1$), and rows $b$ from bottom-to-top in the $y$-coordinate ($1\leq b\leq r-1$), they prove the following result, which they call a ``curious symmetry."
\begin{theorem} \label{AmLev} [Amdeberhan-Leven, Theorem 2.2, \cite{A-L}] For $r\geq 3$ the $(a,b)$ entry of $R$ is an element of $P_{r,r+2}$ if and only if $\{a,r-1-b\}$ is not.
Equivalently, for $1\leq a\leq r+1$ and $1\leq b\leq r-1,$ $(r+1)(b-1)+a\in P_{r,r+2}$ if and only if $(r+1)(r-1-b)+a\not\in P_{r,r+2}.$
\end{theorem} 
\begin{figure}   
\begin{center}
\raisebox{.05in}{\scalebox{0.35}{\begin{tikzpicture}
\draw (0,-13.5cm) node [rectangle, minimum size=1.5cm, inner sep=0pt, draw, anchor=south west]  {\Huge };
\draw (1.5cm,-13.5cm) node [rectangle, minimum size=1.5cm, inner sep=0pt, draw, anchor=south west]  {\Huge };
\draw (3.0cm,-13.5cm) node [rectangle, minimum size=1.5cm, inner sep=0pt, draw, anchor=south west]  {\Huge };
\draw (4.5,-13.5cm) node [rectangle, minimum size=1.5cm, inner sep=0pt, draw, anchor=south west]  {\Huge };
\draw (6.0cm,-13.5cm) node [rectangle, minimum size=1.5cm, inner sep=0pt, draw, anchor=south west]  {\Huge };
\draw (7.5cm,-13.5cm) node [rectangle, minimum size=1.5cm, inner sep=0pt, draw, anchor=south west]  {\Huge };
\draw (9.0cm,-13.5cm) node [rectangle, minimum size=1.5cm, inner sep=0pt, draw, anchor=south west]  {\Huge $47$};
\draw (10.5cm,-13.5cm) node [rectangle, minimum size=1.5cm, inner sep=0pt, draw, anchor=south west]  {\Huge };
\draw (0,-15.0cm) node [rectangle, minimum size=1.5cm, inner sep=0pt, draw, anchor=south west]  {\Huge 33};
\draw (1.5cm,-15.0cm) node [rectangle, minimum size=1.5cm, inner sep=0pt, draw, anchor=south west]  {\Huge };
\draw (3.0cm,-15.0cm) node [rectangle, minimum size=1.5cm, inner sep=0pt, draw, anchor=south west]  {\Huge };
\draw (4.5,-15.0cm) node [rectangle, minimum size=1.5cm, inner sep=0pt, draw, anchor=south west]  {\Huge };
\draw (6.0cm,-15.0cm) node [rectangle, minimum size=1.5cm, inner sep=0pt, draw, anchor=south west]  {\Huge };
\draw (7.5cm,-15.0cm) node [rectangle, minimum size=1.5cm, inner sep=0pt, draw, anchor=south west]  {\Huge $38$};
\draw (9.0cm,-15.0cm) node [rectangle, minimum size=1.5cm, inner sep=0pt, draw, anchor=south west]  {\Huge };
\draw (10.5cm,-15.0cm) node [rectangle, minimum size=1.5cm, inner sep=0pt, draw, anchor=south west]  {\Huge $40$};
\draw (0cm,-16.5cm) node [rectangle, minimum size=1.5cm, inner sep=0pt, draw, anchor=south west]  {\Huge };
\draw (1.5cm,-16.5cm) node [rectangle, minimum size=1.5cm, inner sep=0pt, draw, anchor=south west]  {\Huge };
\draw (3.0cm,-16.5cm) node [rectangle, minimum size=1.5cm, inner sep=0pt, draw, anchor=south west]  {\Huge };
\draw (4.5,-16.5cm) node [rectangle, minimum size=1.5cm, inner sep=0pt, draw, anchor=south west]  {\Huge };
\draw (6.0cm,-16.5cm) node [rectangle, minimum size=1.5cm, inner sep=0pt, draw, anchor=south west]  {\Huge $29$};
\draw (7.5cm,-16.5cm) node [rectangle, minimum size=1.5cm, inner sep=0pt, draw, anchor=south west]  {\Huge };
\draw (9.0cm,-16.5cm) node [rectangle, minimum size=1.5cm, inner sep=0pt, draw, anchor=south west]  {\Huge $31$};
\draw (10.5cm,-16.5cm) node [rectangle, minimum size=1.5cm, inner sep=0pt, draw, anchor=south west]  {\Huge };
\draw (0,-18.0cm) node [rectangle, minimum size=1.5cm, inner sep=0pt, draw, anchor=south west]  {\Huge $17$};
\draw (1.5cm,-18cm) node [rectangle, minimum size=1.5cm, inner sep=0pt, draw, anchor=south west]  {\Huge $18$};
\draw (3.0cm,-18cm) node [rectangle, minimum size=1.5cm, inner sep=0pt, draw, anchor=south west]  {\Huge $19$};
\draw (4.5,-18cm) node [rectangle, minimum size=1.5cm, inner sep=0pt, draw, anchor=south west]  {\Huge $20$};
\draw (6.0cm,-18cm) node [rectangle, minimum size=1.5cm, inner sep=0pt, draw, anchor=south west]  {\Huge };
\draw (7.5cm,-18cm) node [rectangle, minimum size=1.5cm, inner sep=0pt, draw, anchor=south west]  {\Huge $22$};
\draw (9.0cm,-18cm) node [rectangle, minimum size=1.5cm, inner sep=0pt, draw, anchor=south west]  {\Huge };
\draw (10.5cm,-18cm) node [rectangle, minimum size=1.5cm, inner sep=0pt, draw, anchor=south west]  {\Huge $24$};
\draw (0,-19.5cm) node [rectangle, minimum size=1.5cm, inner sep=0pt, draw, anchor=south west]  {\Huge };
\draw (1.5cm,-19.5cm) node [rectangle, minimum size=1.5cm, inner sep=0pt, draw, anchor=south west]  {\Huge $10$};
\draw (3.0cm,-19.5cm) node [rectangle, minimum size=1.5cm, inner sep=0pt, draw, anchor=south west]  {\Huge $11$};
\draw (4.5,-19.5cm) node [rectangle, minimum size=1.5cm, inner sep=0pt, draw, anchor=south west]  {\Huge $12$};
\draw (6.0cm,-19.5cm) node [rectangle, minimum size=1.5cm, inner sep=0pt, draw, anchor=south west]  {\Huge $13$};
\draw (7.5cm,-19.5cm) node [rectangle, minimum size=1.5cm, inner sep=0pt, draw, anchor=south west]  {\Huge };
\draw (9.0cm,-19.5cm) node [rectangle, minimum size=1.5cm, inner sep=0pt, draw, anchor=south west]  {\Huge $15$};
\draw (10.5cm,-19.5cm) node [rectangle, minimum size=1.5cm, inner sep=0pt, draw, anchor=south west]  {\Huge };
\draw (0,-21.0cm) node [rectangle, minimum size=1.5cm, inner sep=0pt, draw, anchor=south west]  {\Huge $1$};
\draw (1.5cm,-21.0cm) node [rectangle, minimum size=1.5cm, inner sep=0pt, draw, anchor=south west]  {\Huge $2$};
\draw (3.0cm,-21.0cm) node [rectangle, minimum size=1.5cm, inner sep=0pt, draw, anchor=south west]  {\Huge $3$};
\draw (4.5,-21.0cm) node [rectangle, minimum size=1.5cm, inner sep=0pt, draw, anchor=south west]  {\Huge $4$};
\draw (6.0cm,-21.0cm) node [rectangle, minimum size=1.5cm, inner sep=0pt, draw, anchor=south west]  {\Huge $5$};
\draw (7.5cm,-21.0cm) node [rectangle, minimum size=1.5cm, inner sep=0pt, draw, anchor=south west]  {\Huge $6$};
\draw (9.0cm,-21.0cm) node [rectangle, minimum size=1.5cm, inner sep=0pt, draw, anchor=south west]  {\Huge };
\draw (10.5cm,-21.0cm) node [rectangle, minimum size=1.5cm, inner sep=0pt, draw, anchor=south west]  {\Huge $8$};
\end{tikzpicture}}}
\caption{Amdeberhan-Leven rectangle $R$ for $P_{7,9}$}
\end{center}
\end{figure}
[There is a precedent for the case Amdeberhan-Leven consider. For $r=2k+1>1$, the maximal simultaneous $(r,r+2)$-core is self-conjugate by Theorem \ref{s2t2}. The author and C. Hanusa showed in \cite{H-N} that it is more natural to think about simultaneous $(r,r+2)$-core partitions than simultaneous $(s,s+1)$-core partitions, which behave better in the non-self-conjugate case.] For the remainder of this paper we will let $s=r+1$, and will consider maximal $(s-1,s+1)$-core, where $s$ is even and greater than 2. We now review the $s$-abacus, $s$-core, and $s$-quotient constructions.
\subsection{bead-sets}
A {\it bead-set} $X$ corresponding to a partition $\lambda$ is generalization of
the set of first column hooks in the following sense:
$X=\{0,1,\cdots,k,|h_{11}|+k,|h_{12}|+k, |h_{13}|+k,\cdots\}$ for some non-negative
integer $k$. It can also be seen as a finite set of non-negative integers, represented by {\it beads} at
integral points of the $x$-axis, i.e. a bead at position $x$ for
each $x$ in $X$ and {\it spacers} at positions not in $X$. A {\it minimal} bead-set $X$ is one where the first space is counted as 0. Then
$|X|$ is the number of beads that occur after the zero position, where ever that may fall. We say $X=\{0,1,\cdots,k,|h_{11}|+k,|h_{12}|+k, |h_{13}|+k,\cdots\}$ is {\it normalized with respect to s} if $k$ is the minimal integer such that $|X|\equiv 0 \pmod s.$ 
\subsection{2-cores and staircase partitions}
[The results in this section are stated without proof; for more details see Section 2 in \cite{O}.] The set of hooks $\{h_{\iota \gamma}\}$ of $\lambda$ correspond bijectively to pairs $(x,y)$ where $x\in X$, $y\not\in X$ and $x>y$; that is, a bead in the bead-set $X$ of $\lambda$ and a spacer to the left of it. Hooks of length $s$ are those such that $x-y=s$. 

Each first-column hook length, or bead $x_i$ in the minimal bead-set $X$, also corresponds to a row, or component $\lambda_i$ of $\lambda$ The following result allows us to recover the size of the components from $X.$
\begin{lemma} \label{onebead}The size of the component $\lambda_i$ corresponding to the bead $x_i\in X$ is the number of spacers to the left of the bead; that is, $\lambda_i=|y\not\in X:y<x_i|.$
\end{lemma}
Let $\tau_{k}=(k,k-1,\cdots,1)$ be the {\it k-th staircase partition}. Then $|\tau_k|=t_k$ where $t_{k}=\binom{k+1}{2}$ (the $k$-th triangular number). The following lemmas are well-known.
\begin{lemma} \label{2core}The 2-core partitions are exactly the staircase partitions.
\end{lemma}
\begin{lemma} \label{2coreab}The minimal $X$ for the 2-core $\tau_k$ is $\{1,3,5,\cdots, 2k-3, 2k-1\}.$ In other words, the 2-core partitions are sequence of alternating spacers-and-beads of length $2k-1$.
\end{lemma}
\subsection{The $s$-abacus}
Given a fixed integer $s$, we can arrange the nonnegative integers
in an array of columns and consider the columns as runners.
\vs
\[
\begin{array}{cccc}
ms &  & &  (m+1)s-1\\
\vdots  &   & \ddots &\\
s      &  s+1   &  &  2s-1\\
0      &  1     & \cdots &  s-1
\end{array}
\vs
\]
The column containing $i$ for $0\leq i \leq{s-1}$ will be called {\it runner i}. The positions $0,1,2,\cdots$ on the $i$th runner corresponding to $i,i+s,i+2s,\cdots$ will be called {\it i-positions}. Placing a bead at position $x_{j}$ for each $x_{j}\in{X}$ gives the $\textit{s-abacus diagram}$ of $X$. A {\it normalized} abacus will be one whose bead-set $X$ is normalized, a {\it minimal} abacus is one in which $X$ is minimal (or, the first spacer is counted as the zero position).
\subsection{The $s$-core and $s$-quotient}
By removing a sequence of $s$-hooks from $\lambda$ until no $s$-hooks remain, one obtains its {\it s-core} $\lambda^0$. The $s$-abacus of $\lambda^0$ can be found from the $s$-abacus of $\lambda$ by pushing beads in each runner down as low as they can go (Theorem 2.7.16,\cite{J-K}: we have changed the orientation). This implies $\lambda^0$ is unique since it is independent of the way the $s$-hooks are removed.
\vs
For $0\leq{i}\leq{p-1}$ let $X_{i}=\{j:i+js\in{X}\}$ and let $\lambda_{(i)}$ be the
partition represented by the bead-set $X_{i}$.  The
{\it s-quotient} of $\lambda$ is the sequence $(\lambda_{(0)},\cdots,\lambda_{(s-1)})$ obtained from $X$. The next lemma is Proposition 3.5 in \cite{O}.
\begin{lemma} \label{obij}
Let $\lambda$ be a partition with $s$-core $\lambda^0$ and $s$-quotient $(\lambda_{(i)})$, $0\leq i \leq s-1$. Then
\begin{enumerate}
\item Every 1-hook in $\lambda_{(i)}$ corresponds to a $s$-hook in $\lambda$ for $0\leq i\leq s-1.$
\item $n=|\lambda^0|+\sum_{i} |\lambda_{(i)}|.$
\end{enumerate}
\end{lemma}
Lemma \ref{obij} implies that there exists a bijection between a partition $\lambda$ and its $s$-core and $s$-quotient, such that each node in some $\lambda_{i}$ corresponds to an $s$-hook in $\lambda.$ The situation is strengthened when $\lambda$ is self-conjugate.
\begin{lemma}\label{sym_quo} 
Suppose $|X|=0\pmod{s}.$ Let $\lambda^*$ be the conjugate of $\lambda$, $(\lambda^*)^0$ its $s$-core and let $(\lambda_i^*)$ be the $s$-quotient of $\lambda^*$, $0\leq i\leq s-1$. Then
\begin{enumerate}
\item $(\lambda^*)^0=(\lambda^0)^*$
\item $(\lambda_{(i)})^*=\lambda_{(s-1-i)}$.
\end{enumerate}
In particular, $\lambda=\lambda^{*}$ if and only if $\lambda^0=(\lambda^0)^*$ and $(\lambda_{(i)})^*=(\lambda^*)_{(i)}.$
\end{lemma}
\subsection{The axis of symmetry}
The following results and their proofs can be found in Section 4, \cite{N}.
\begin{proposition}\label{halfint} Suppose $\lambda$ is a partition of $n$ and
let $X$ be a bead-set for $\lambda$. Then there exists a
half-integer $\theta(\lambda)$ such that the number of beads to the
right of $\theta(\lambda)$ equals the number of spaces to its left.  Conversely, given a bead-spacer sequence and a half-integer $\theta(\lambda)$ such that the number of beads to the right equals the number of spaces to the left, one can recover the unique partition $\lambda.$
\end{proposition}
\noindent
\begin{lemma} Let $X$ be a minimal bead-set for $\lambda$. If $x'\in X$ is the entry with maximum value, $\theta(\lambda)=\frac{x'}{2}.$
\end{lemma}
We call $\theta(\lambda)$ the {\it axis} of $\lambda.$ If $\lambda$ is self-conjugate we say $X$ has a {\it axis of
symmetry}. 
\begin{corollary} \label{selfhalfint} Le $X$ be a bead-set for $\lambda$. Then $\lambda$ is a self-conjugate partition if and only if there exists a half-integer $\theta(\lambda)$
such that beads and spaces in $X$ to the right of $\theta(\lambda)$ are reflected respectively to spaces and beads in $X$ to its left. 
\end{corollary}
When $\lambda^0=\emptyset,$ each $\lambda_{i}$ has an axis of symmetry $\theta(\lambda_{i})$ induced by $X$.
\begin{lemma} \label{selfquohalfint} Suppose
$X$ is normalized. Then $|X|=ms,$ $\lambda^0=\emptyset$, and each runner has exactly $m$ beads if and only if
$\theta(\lambda_{(i)})=\theta(\lambda_{(i')})=m-\frac{1}{2}$ for all $0\leq i,i' \leq s-1.$ 
\end{lemma}
\begin{example} The maximum $(5,7)$-core $\kappa_{5,7}$ has empty 8-core. In the normalized (minimal) 8-abacus in {\bf Figure 1}, each $\lambda_{(i)}$ has axis $\theta(\lambda_{(i)})=\frac{5}{2}.$  
\end{example}
\section{The $s$-quotient of $\kappa_{s\pm1}$}
\subsection{The $s$-abacus of $\kappa_{s\pm 1}$}
We begin with a classical result of Sylvester.
\begin{lemma} \label{syl} The largest integer in $P_{s,t}$ is $st-s-t.$
\end{lemma}
The Amdeberhan-Leven rectangle $R$ is constructed to begin at 0; the $(r+1)$-abacus of $\kappa_{r,r+2}$ starts at 0.  However $0\not\in P_{r,r+2}$ and by Lemma \ref{syl} neither is $(r+1)(r-1)$. Hence $R$ and the minimal $(r+1)$-abacus of $\kappa_{r,r+2}$ include the same values. 

Recall $s=r+1.$ We now interpret the Amdeberhan-Leven result in terms of the $s$-abacus $\kappa_{s-1,s+1}$. We use a runner-row index. We start with a definition.
\begin{definition}\label{pi}
Let $s=2k>2$. Then $\alpha(s)$ is an $s$-abacus with $s$ runners, indexed from left-to-right by $0\leq i\leq s-1$ and $s-2$ rows, indexed from bottom-to-top by $0\leq i\leq s-3$, which is  constructed as follows: For each $i\in[0,k-2]$, the runners $i$ and $2k-i-1$ are composed firstly of beads in rows $j$ where $0\leq j\leq i$. Then rows $j>i$ consist of alternating spacers-and-beads, until the total number of beads in each runner is $(k-1)$. Spacers fill the remainder of the rows. 
\end{definition}
\begin{example} $\alpha(8)$ has three beads in each runner. Runners $i$=3 and 4 consist of three beads below three spacers; $i=$2 and 5 have two beads followed by a spacer-and-bead, then two spacers; $i$=1 and 6 have one bead followed by spacer-bead-spacer-bead-spacer; and runners $i$=0 and 7 have an alternating sequence of spacers-and-beads. [See {\bf Figure 1}.]
\end{example}
\begin{lemma}
The $s$-abacus $\alpha(s)$ is normalized with respect to $s$.
\end{lemma}
\begin{proof} The total number of beads in $\alpha(s)$ is $2k(k-1)=\frac{s^2-2s}{2}$, a multiple of $s$.
\end{proof}
\begin{lemma} \label{rowj} Fix $1<j<2k-3$ and $0\leq i< k-1.$
\begin{enumerate} 
\item There is a bead in row $j$ of runner $0$ if and only if there is a bead in row $j-1$ of runner $1$.
\item There is a bead in row $j$ of runner $2k-1$ if and only if there is a bead in row $j-1$ of runner $2k-2$.
\item There is a spacer in row $j$ of runner $0$ if and only if there is a spacer in row $j+1$ of runner $1$.
\item There is a spacer in row $j$ of runner $2k-1$ if and only if there is a spacer in row $j$ of runner $2k-2$.
\end{enumerate}
\end{lemma}
\begin{proof} By Definition \ref{pi}, runner $i=0$ begins in row $j=0$ with a spacer, and continues upwards with alternating beads-and-spacers. Runner $i=1$ begins with a bead in row 1, and continues upwards, alternating spacers-and-beads. Since both columns have $2k-2$ rows, (1) and (3) follow. For (2) and (4), a similar argument holds.
\end{proof}
\begin{lemma} \label{nest}
The $s$-abacus $\alpha(s+2)$ can be obtained from the $s$-abacus $\alpha(s)$ using the following procedure:
\begin{enumerate}
\item Append a new row of $2k$ beads below $\alpha(s).$
\item Append a new row of $2k$ spacers above $\alpha(s).$
\item Append a new runner of length $2k-2$ consisting of alternating beads-and-spacers to the {\it left} (and an identical column to the {\it right}) of $\alpha(s).$ 
\item Append a single spacer to the bottom, and a single bead at the top of, both new runners in step (3). [The total number of beads in all runners, both the two new runners, as well as the $s=2k$ previous runners, will now be $k.$]
\item Renumber the runners with $i'$ so $0\leq i' \leq 2k+1$ and the rows with $j'$ so that $0\leq j'\leq 2k-1.$ Renumber the abacus positions, with 0 in the bottom left-most corner, increasing from left-to-right and bottom-to-top, with final position $(2k+1)(2k-1)$ in the upper-right-hand corner.
\end{enumerate}
\end{lemma}
\begin{proof} It is enough to see that the outcome satisfies Definition \ref{pi} for $\alpha(s+2).$
\end{proof}
\begin{example} To see how Lemma \ref{nest} is used to obtain $\alpha(10)$ from $\alpha(8)$, see {\bf Appendix A, Figure 9} and {\bf Figure 8}.
\end{example}
Recall $\lambda^0$ is the $s$-core partition of $\lambda,$ $(\lambda_{(i)})$ is the $s$-quotient (where $0\leq i\leq s-1$), and that $\tau_{\ell}$ the the ${\ell}$-th 2-core partition. For the following two lemmas we abuse notation and let $\alpha(s)$ refer not only to the $s$-abacus but also to its corresponding partition.
\begin{lemma} \label{tau} Suppose $s=2k>2.$ Then
\begin{enumerate}
\item $\alpha(s)^0=\emptyset$
\item $\alpha(s)_{(i)}=\alpha(s)_{(s-i-1)}=\tau_{k-i+1}.$
\end{enumerate}
\end{lemma}
\begin{proof} We proof each condition separately.
\begin{enumerate}
\item Since each runner $\alpha(s)_{i}$ has $k-1$ beads and $(k-1)$ spacers, the removal of all $s$-hooks will result in an $s$-abacus with each runners having $k-1$ beads beneath $k-1$ spacers. This corresponds to the empty partition. 
\item We use induction on $k$. For $k=2$ it is true. Assume it is for $k.$ We obtain the $\alpha(s+2)$ from $\alpha(s)$ by Lemma \ref{pi}. By construction, for $1\leq i'\leq 2k$ we have $|\alpha(s)_{(i'-1)}|=|\alpha(s+2)_{(i')}|$; hence, by the inductive hypothesis and since $i+1=i'$, $|\alpha(s+2)_{(i')}|=\tau_{(k+1)-i'-1}$. It only remains to check $i'=0,2k+1.$ The proof is finished using (3) and (4) of Lemma \ref{nest} and Lemma \ref{2coreab}. 
\end{enumerate}
\end{proof}
\begin{example} $\alpha(8)$ has 8-quotient $(\lambda_0,\cdots,\lambda_{s-1})$
\[
(3,2,1),(2,1),(1),\emptyset,\emptyset,(1),(2,1),(3,2,1))
\] 
[See {\bf Appendix A, Figure 8} and {\bf Appendix B, Figure 12}]
\end{example}
\begin{lemma} \label{kappaispi} Let $s=2k>2$. Then $\alpha(s)$ is the minimal $s$-abacus for $\kappa_{s-1,s+1}.$
\end{lemma}
\begin{proof}
By construction, $\alpha(s)$ is minimal, since the first spacer labels zero. 
We must show:
\begin{enumerate}
\item $|\alpha(s)|=\frac{((2k-1)^2-1)((2k+1)^2-1)}{24}$;
\item $\alpha(s)$ contains no $(2k-1)$-hooks or $(2k+1)$-hooks.
\end{enumerate}
Then by the uniqueness implied by Theorem \ref{s2t2}, $\alpha(s)=\kappa_{s\pm1}$. We use the structure of $\alpha(s)$ and induction on $k$.

\vs
By Theorem \ref{obij} each 1-hook in the $s$-quotient corresponds to a $s$-hook in $\lambda$. Hence, to prove (1), since $\alpha(s)^0=\emptyset$, it is enough to calculate $\sum_i|\lambda_{(i)}|$ and multiply by $s=2k$. This equals $2k\cdot 2\sum^{k-1}_{i=1}t_{i}=(4k)\frac{(k-1)(k)(k+1)}{6}$. In particular $4k\frac{(k)(k^2-1)}{6}=\frac{16k^4-16k^2}{24}=\frac{(4k^2-4k)(4k^2+4k)}{24}$, which, after completing-the-square, equals to $\frac{((2k-1)^2-1)((2k+1)^2-1)}{24}$. We are done.

\vs
To prove (2), we use induction on $k>2$. For the basic case, $s$=4, it holds: $ \alpha(4)$ has no 3-hooks or 5-hooks. [See {\bf Appendix A, Figure 6}.]

\vs
By the inductive hypothesis we know the $2k$-abacus of $\kappa_{2k\pm1}$ contains no $(2k-1)$-hooks or $(2k+1)$-hooks. More specifically, no bead in $\alpha(s)$ has a spacer either $2k+1$ or $2k-1$ positions below it. Apply Lemma \ref{nest} to obtain $\alpha(s+2)$; this adds two additional positions between the beads and spacers arising from $\alpha(s)$. Hence there are no $(2k+1)$-hooks or $(2k+3)$-hooks arising from bead-spacer pairs $(x,y)$ where both $x$ and $y$ are in runners $1<i'<2k-2.$ It remains to examine the beads and spacers introduced by runners $i'=0,2k+1$. 

\vs
If a bead in row $j'$ of runner $i'=0$ were to add a new $(2k+3)$-hook, a spacer would have to appear in row $j'-2$ of the runner $i'=2k+1$. By construction, such positions are occupied by beads, since runners $0$ and $2k+1$ are identical. If a bead in row $j'$ of $i'=0$ were to add a new $(2k+1)$-hook, a spacer would have to appear in row $j'-1$ of runner $i'=1$. But by the Lemma \ref{rowj}(1), this position is always occupied by a bead.

\vs
If a bead in row $j'$ on runner $i'=2k+1$ were to add a new $(2k+3)$-hook, a spacer would appear in row $j'-1$ of runner $i'=2k$. But by Lemma \ref{rowj}(2) this position is always occupied by a bead. If a bead in row $j'$ of runner $i'=2k+1$ were to add a new $(2k+1)$-hook, a spacer would have to appear in the same row in the runner $i'=0$. By construction, the two runners are identical, so a bead in one implies a bead in the other.  

If a spacer in row $j'$ of runner $i'=0$ were to add a new $(2k+3)$-hook, a bead would have to appear in row $j'+1$ of runner $i'=1.$ But by Lemma \ref{rowj}(3), this position is always occupied by a spacer. If a spacer in row $j'$ of $i'=0$ were to add $(2k+1)$-hook, a bead would have to appear in the same row of runner $i'=2k+1.$ By construction, the two runners are identical, so a spacer in one implies a spacer in the other.

If a spacer in row $j'$ of runner $i'=2k+1$ were to add a new $(2k+3)$-hook, a bead would have to appear in row $j'+2$ in runner $i'=0$; by construction, since both runners are identical alternating sequences of spacer-and-beads, such positions are occupied by spacers. If a spacer in row $j'$ of runner $i'=2k+1$ were to add a new $(2k+1)$-hook, a bead would have to appear in row $j'+1$ of runner $i'=2k.$ But by Lemma \ref{rowj}(4) this position is occupied by a spacer.

\end{proof}
\subsection{An alternative proof of Amdeberhan-Leven}
Using the results of this section we offer an alternative proof to Theorem \ref{AmLev}.
\begin{proof}[Proof of Theorem \ref{AmLev}] By Lemma \ref{kappaispi}, the $s$-core of $\kappa_{s\pm1}=\emptyset,$ and $X$ is normalized. Again by Lemma \ref{kappaispi}, each $\lambda_{i}$ is self-conjugate, so each runner obeys Lemma \ref{selfhalfint}. By Lemma \ref{selfquohalfint}, all $(\kappa_{s-1,s+1})_i$ have the same axis of symmetry, which is at $i$-position $\frac{s-3}{2}$.  Our runner-row index is $0\leq j\leq s-3$ with $s=r-1,$ which finishes the proof.
\end{proof}
\section{Generalizations}
\subsection{Additional symmetry}
Using Theorem \ref{piquo} we can strengthen Amdeberhan-Leven to include additional symmetry. 
\begin{theorem} Let $s=2k>2$ and let $\alpha(s)$ be the $s$-abacus of $\kappa_{s\pm1}.$ Then the following are equivalent:  
\begin{enumerate}
\item $(i,j)\in \alpha(s)$
\item $(i,s-3-j)\not\in \alpha(s)$
\item $(s-1-i, j)\in \alpha(s).$
\end{enumerate}
\end{theorem}
\begin{proof} By Theorem \ref{AmLev} is sufficient to prove $(1)\iff (3).$ This follows from Lemma \ref{kappaispi} and Lemma \ref{tau}.
\end{proof}
\subsection{Horizontal anti-symmetry and vertical symmetry}
The symmetries exhibited by the $s$-abacus of $\kappa_{s\pm1}$ can be formalized and generalized to a larger family of partitions. For the remainder of this section we assume that the bead-set $X$ of $\lambda$ is normalized with respect to $s$. Suppose that the $s$-abacus of $\lambda$ has maximum value $i+(q-1)s$. In particular, the $s$-abacus of $\lambda$ has $s$ columns and $q$ rows. 
\begin{definition} \label{hori}
We say the $s$-abacus of $\lambda $ exhibits {\it horizontal anti-symmetry} if a there is a bead in the $(i,j)$th-position if and only if there is a spacer in the $(i,q-j-1)$ position. 
\end{definition}
\begin{definition}\label{vert}
We say the $s$-abacus of $\lambda$ exhibits {\it vertical symmetry} if there is a bead in the $(i,j)$th-position if and only if there is a bead in the $(s-i-1,j)$th-position. 
\end{definition}
\begin{lemma} \label{horiz} The $s$-abacus of $\lambda$ exhibits horizontal anti-symmetry if and only if $q$ is even, $\lambda_{(i)}=\lambda^*_{(i)},$ and each runner has $q$ beads.
\end{lemma}
\begin{proof} Suppose the $s$-abacus of $\lambda$ exhibits horizontal anti-symmetry. Clearly $q$ must be even, otherwise there would exist a bead or spacer in a row that would not have a spacer or bead to pair with. Let $q=2m.$ Horizontal symmetry also implies each runner $i$ must have the same axis of symmetry, $\theta(\lambda_i)=\frac{q-1}{2}=m-\frac{1}{2}$. Lemma \ref{selfhalfint} implies $\lambda=\lambda^*$. By Lemma \ref{selfquohalfint}, each runner has exactly $q$ beads.  The proof in the other direction is clear.
\end{proof}
\begin{lemma} \label{verti} The $s$-abacus of $\lambda $ exhibits vertical symmetry if and only if $s$ is even, runner $i$ and runner $s-i-1$ have the same number of beads, and $\lambda_i=\lambda_{s-i-1}$ for $0\leq i \leq s-1.$
\end{lemma}
\begin{proof} Suppose the $s$-abacus of $\lambda$ exhibits horizontal symmetry. Then $s$ must be even, otherwise there would a bead or spacer in a runner that would not have a bead or spacer to pair with. Vertical symmetry also implies that each runner $i$ and $s-i-1$ must be identical. This means runners $i$ and $s-i-1$ have the same number of beads and $\lambda_i=\lambda_{s-i-1}$ for each $0\leq i\leq s-1$. The proof in the other direction is clear.
\end{proof}
\begin{theorem} \label{strongs} $\lambda $ exhibits both horizontal anti-symmetry and vertical symmetry with respect to $s$ if and only if $s$ and $q$ are both even and the following three conditions hold for all $0\leq i\leq s-1$
\begin{enumerate}
\item $\lambda^0=\emptyset$  
\item $\lambda_{(i)}=\lambda^*_{(i)}$
\item $\lambda_{(i)}=\lambda_{(s-i-1)}.$
\end{enumerate}
\end{theorem}
\begin{proof} This follows from Lemma \ref{horiz} and Lemma \ref{verti}. 
\end{proof}
\begin{example} $\lambda=(8,6^4,1^2)$ exhibits horizontal anti-symmetry and vertical symmetry with respect to $s=4$, but is neither a 3-core nor a 5-core. See {\bf Figure 5}.
\end{example}
The following corollary is immediate.
\begin{corollary} Let $s=2k>1.$ The $s$-abacus of $\kappa_{s\pm1}$ exhibits horizontal anti-symmetry and vertical symmetry.
\end{corollary}
\begin{corollary} If the $s$-abacus of $\lambda$ exhibits both horizontal anti-symmetry and vertical symmetry then $\lambda$ is self-conjugate.
\end{corollary}
\begin{proof} By Theorem \ref{strongs}, since $\lambda_{(i)}=\lambda_{(s-i-1)}$ and $\lambda_{(i)}=\lambda^*_{(i)},$ we have $\lambda_{(i)}=\lambda^*_{(s-i-1)}.$ Since $\lambda^0=\emptyset,$ and by assumption $|X|=0\pmod{s},$ we have $\lambda=\lambda^*$ by Lemma \ref{sym_quo}.
\end{proof}
\begin{figure}
\caption{The minimal $4$-abacus of $\lambda=(8,6^4,1^2)$ (see Example 4.7) }
\hspace{-1cm}
\begin{center}
\scalebox{0.3}{\begin{tikzpicture}
\vs

\draw (4cm,-2cm) node [circle, minimum size=1.5cm, inner sep=0pt, draw, anchor=south west] {\Huge $12$};
\draw (6cm,-2cm) node [circle, minimum size=1.5cm, inner sep=0pt, draw=none, anchor=south west] {\Huge $13$};
\draw (8cm,-2cm) node [circle, minimum size=1.5cm, inner sep=0pt, draw=none, anchor=south west] {\Huge $14$};
\draw (10cm,-2cm) node [circle, minimum size=1.5cm, inner sep=0pt, draw, anchor=south west] {\Huge $15$};
\draw (4cm,-4cm) node [circle, minimum size=1.5cm, inner sep=0pt, draw, anchor=south west] {\Huge $8$};
\draw (6cm,-4cm) node [circle, minimum size=1.5cm, inner sep=0pt, draw, anchor=south west] {\Huge $9$};
\draw (8cm,-4cm) node [circle, minimum size=1.5cm, inner sep=0pt, draw, anchor=south west] {\Huge $10$};
\draw (10cm,-4cm) node [circle, minimum size=1.5cm, inner sep=0pt, draw, anchor=south west] {\Huge $11$};
\draw (4cm,-6cm) node [circle, minimum size=1.5cm, inner sep=0pt, draw=none, anchor=south west] {\Huge $4$};
\draw (6cm,-6cm) node [circle, minimum size=1.5cm, inner sep=0pt, draw=none, anchor=south west] {\Huge $5$};
\draw (8cm,-6cm) node [circle, minimum size=1.5cm, inner sep=0pt, draw=none, anchor=south west] {\Huge $6$};
\draw (10cm,-6cm) node [circle, minimum size=1.5cm, inner sep=0pt, draw=none, anchor=south west]{\Huge $7$};
\draw (4cm,-8cm) node [circle, minimum size=1.5cm, inner sep=0pt, draw=none, anchor=south west] {\Huge $0$};
\draw (6cm,-8cm) node [circle, minimum size=1.5cm, inner sep=0pt, draw, anchor=south west] {\Huge $1$};
\draw (8cm,-8cm) node [circle, minimum size=1.5cm, inner sep=0pt, draw, anchor=south west] {\Huge $2$};
\draw (10cm,-8cm) node [circle, minimum size=1.5cm, inner sep=0pt, draw=none, anchor=south west]{\Huge $3$};
\end{tikzpicture}}
\end{center}
\end{figure}
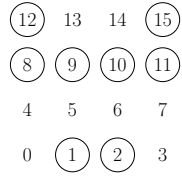
\section{Further Study}
\subsection{Simultaneous ${\bf (s-1,s,s+1)}$-cores}
The following theorem is a recently-proven conjecture of Amdeberhan \cite{A}.
\begin{theorem} (Yang-Zhong-Zhou, \cite{Y-Z-Z}; H. Xiong, \cite{X}) \label{extra-s} The size of the largest $(s-1,s,s+1)-$core is
\begin{enumerate}
\item $k\binom{k+1}{3}$ if $s=2k>2$
\item $(k+1)\binom{k+1}{3}$+$\binom{k+2}{3}$ if $s=2k+1>2.$
\end{enumerate}
\end{theorem}
Let $\kappa_{s-1,s,s+1}$ is a (not necessarily unique) simultaneous $(s-1,s,s+1)$-core of maximal size. Theorem \ref{extra-s} allows us to compare $|\kappa_{s\pm1}|$ with $|\kappa_{s-1,s,s+1}|.$ 
\begin{proposition}\label{extra4} Let $s=2k>2.$ Then $|\kappa_{s\pm1}|>|\kappa_{(s-1,s,s+1)}|.$ In particular, $|\kappa_{s\pm1}|=4|\kappa_{(s-1,s,s+1)}|$
\end{proposition}
\begin{proof} Since $s$ is even, by Theorem \ref{extra-s}(1) above
$|\kappa_{s-1,s,s+1}|=\frac{k^4-k^2}{6}$. However by Theorem \ref{s2t2}, $|\kappa_{s\pm 1}|=\frac{((s-1)^2-1)((s+1)^2-1)}{24}.$ This simplifies to $\frac{4(k^4-k^2)}{6}.$ The result follows.
\end{proof}
\begin{corollary}\label{nevers}
$\kappa_{s,s+2}$ is never an $s$-core. 
\end{corollary}
Corollary \ref{nevers} also follows from Theorem \ref{piquo} which says $\kappa_{s,s+2}$ is comprised completely of $s$-hooks. 
Is there interpretation (either in the geometry of the $s$-abacus or in the manipulation of Young diagrams) of the factor of 4 that appears above? A cursory examination of $\kappa_{(3,5)}$ and $\kappa_{(3,4,5)}$ does not suggest an obvious one.
\subsection{Other proofs using the $s$-abacus}
In their proof of Theorem \ref{AmLev} Amdeberhan and Leven use the following result (Corollary 2.1 (ii), \cite{A-L}). 
\begin{lemma}\label{half} Exactly half of the integers in $\{1,2,\cdots, (s-1)(t-1)\}$ belong to $P_{s,t}.$
\end{lemma}
They cite a result of T. Popoviciu on the integral and fractional parts of an integer of which this is a consequence. We provide an alternative proof using only the geometry of the $s$-abacus. 
\begin{proof}[Proof of Lemma \ref{half}] Since by Lemma \ref{syl} neither $(s-1)(t-1)$ nor $0$ are in $P_{s,t}$, it is equivalent to prove half of the integers in $\{0,1,2, \cdots, st-s-t\}$ are in the minimal bead-set of $\kappa_{s,t}$. By Lemma \ref{halfint}, the axis is $\theta(\kappa_{s,t})=\frac{st-s-t}{2}$. This implies the result.
\end{proof} 
Perhaps there are other results on simultaneous core partitions that can be understood using bead-sets and the geometry of the $s$-abacus.

\vs\vs\vs\vs
{\bf Acknowledgements} This paper was conceived while visiting the University of Minnesota Duluth REU in July 2014. The author thanks J. Gallian for the invitation and hospitality while there. The author also thanks Christopher R. H. Hanusa for his valuable comments on the manuscript and his suggestions on diagrams.

\clearpage

\vspace{1cm}
\begin{center}
\begin{figure}
{\bf APPENDIX A}\\
 \vspace{0.25cm}
 {\bf The $s$-abaci ${\bf \alpha(s)}$ of ${\bf \kappa_{s\pm1}}$}
\vs\vs
\vspace{.2cm}
\caption{$s=4$}
\hspace{-1cm}
\begin{center}
\scalebox{0.3}{\begin{tikzpicture}

\draw (4cm,-2cm) node [circle, minimum size=1.5cm, inner sep=0pt, draw, anchor=south west] {\Huge $4$};
\draw (6cm,-2cm) node [circle, minimum size=1.5cm, inner sep=0pt, draw=none, anchor=south west] {\Huge $5$};
\draw (8cm,-2cm) node [circle, minimum size=1.5cm, inner sep=0pt, draw=none, anchor=south west] {\Huge $6$};
\draw (10cm,-2cm) node [circle, minimum size=1.5cm, inner sep=0pt, draw, anchor=south west] {\Huge $7$};
\draw (4cm,-4cm) node [circle, minimum size=1.5cm, inner sep=0pt, draw=none, anchor=south west] {\Huge $0$};
\draw (6cm,-4cm) node [circle, minimum size=1.5cm, inner sep=0pt, draw, anchor=south west] {\Huge $1$};
\draw (8cm,-4cm) node [circle, minimum size=1.5cm, inner sep=0pt, draw, anchor=south west] {\Huge $2$};
\draw (10cm,-4cm) node [circle, minimum size=1.5cm, inner sep=0pt, draw=none, anchor=south west] {\Huge $3$};
\end{tikzpicture}}
\end{center}
\end{figure}

\vs\vs\vs\vs\vs
\begin{figure}
\caption{$s=6$}
\hspace{-1.5cm}
\begin{center}
\scalebox{0.3}{\begin{tikzpicture}

\draw (2cm,0cm) node [circle, minimum size=1.5cm, inner sep=0pt, draw, anchor=south west] {\Huge $18$};
\draw (4cm,0cm) node [circle, minimum size=1.5cm, inner sep=0pt, draw=none, anchor=south west] {\Huge $19$};
\draw (6cm,0cm) node [circle, minimum size=1.5cm, inner sep=0pt, draw=none, anchor=south west] {\Huge $20$};
\draw (8cm,0cm) node [circle, minimum size=1.5cm, inner sep=0pt, draw=none, anchor=south west] {\Huge $21$};
\draw (10cm,0cm) node [circle, minimum size=1.5cm, inner sep=0pt, draw=none, anchor=south west] {\Huge $22$};
\draw (12cm,0cm) node [circle, minimum size=1.5cm, inner sep=0pt, draw, anchor=south west] {\Huge $23$};
\draw (2cm,-2cm) node [circle, minimum size=1.5cm, inner sep=0pt, draw=none, anchor=south west] {\Huge $12$};
\draw (4cm,-2cm) node [circle, minimum size=1.5cm, inner sep=0pt, draw, anchor=south west] {\Huge $13$};
\draw (6cm,-2cm) node [circle, minimum size=1.5cm, inner sep=0pt, draw=none, anchor=south west] {\Huge $14$};
\draw (8cm,-2cm) node [circle, minimum size=1.5cm, inner sep=0pt, draw=none, anchor=south west] {\Huge $15$};
\draw (10cm,-2cm) node [circle, minimum size=1.5cm, inner sep=0pt, draw, anchor=south west] {\Huge $16$};
\draw (12cm,-2cm) node [circle, minimum size=1.5cm, inner sep=0pt, draw=none, anchor=south west] {\Huge $17$};
\draw (2cm,-4cm) node [circle, minimum size=1.5cm, inner sep=0pt, draw, anchor=south west] {\Huge $6$};
\draw (4cm,-4cm) node [circle, minimum size=1.5cm, inner sep=0pt, draw=none, anchor=south west] {\Huge $7$};
\draw (6cm,-4cm) node [circle, minimum size=1.5cm, inner sep=0pt, draw, anchor=south west] {\Huge $8$};
\draw (8cm,-4cm) node [circle, minimum size=1.5cm, inner sep=0pt, draw, anchor=south west] {\Huge $9$};
\draw (10cm,-4cm) node [circle, minimum size=1.5cm, inner sep=0pt, draw=none, anchor=south west] {\Huge $10$};
\draw (12cm,-4cm) node [circle, minimum size=1.5cm, inner sep=0pt, draw, anchor=south west] {\Huge $11$};
\draw (2cm,-6cm) node [circle, minimum size=1.5cm, inner sep=0pt, draw=none, anchor=south west] {\Huge $0$};
\draw (4cm,-6cm) node [circle, minimum size=1.5cm, inner sep=0pt, draw, anchor=south west] {\Huge $1$};
\draw (6cm,-6cm) node [circle, minimum size=1.5cm, inner sep=0pt, draw, anchor=south west] {\Huge $2$};
\draw (8cm,-6cm) node [circle, minimum size=1.5cm, inner sep=0pt, draw, anchor=south west] {\Huge $3$};
\draw (10cm,-6cm) node [circle, minimum size=1.5cm, inner sep=0pt, draw, anchor=south west]{\Huge $4$};
\draw (12cm,-6cm) node [circle, minimum size=1.5cm, inner sep=0pt, draw=none, anchor=south west]{\Huge $5$};
\end{tikzpicture}}
\end{center}
\end{figure}

\vs
\begin{figure}
\caption{$s=8$}
\hspace{-1.5cm}
\begin{center}
\scalebox{0.3}{\begin{tikzpicture}
\draw (0cm,-8cm) node [circle, minimum size=1.5cm, inner sep=0pt, draw=none, anchor=south west] {\Huge $0$};
\draw (2cm,-8cm) node [circle, minimum size=1.5cm, inner sep=0pt, draw, anchor=south west] {\Huge $1$};
\draw (4cm,-8cm) node [circle, minimum size=1.5cm, inner sep=0pt, draw, anchor=south west] {\Huge $2$};
\draw (6cm,-8cm) node [circle, minimum size=1.5cm, inner sep=0pt, draw, anchor=south west] {\Huge $3$};
\draw (8cm,-8cm) node [circle, minimum size=1.5cm, inner sep=0pt, draw, anchor=south west] {\Huge $4$};
\draw (10cm,-8cm) node [circle, minimum size=1.5cm, inner sep=0pt, draw, anchor=south west] {\Huge $5$};
\draw (12cm,-8cm) node [circle, minimum size=1.5cm, inner sep=0pt, draw, anchor=south west] {\Huge $6$};
\draw (14cm,-8cm) node [circle, minimum size=1.5cm, inner sep=0pt, draw=none, anchor=south west] {\Huge $7$};
\draw (0cm,-6cm) node [circle, minimum size=1.5cm, inner sep=0pt, draw, anchor=south west] {\Huge $8$};
\draw (2cm,-6cm) node [circle, minimum size=1.5cm, inner sep=0pt, draw=none, anchor=south west] {\Huge $9$};
\draw (4cm,-6cm) node [circle, minimum size=1.5cm, inner sep=0pt, draw, anchor=south west] {\Huge $10$};
\draw (6cm,-6cm) node [circle, minimum size=1.5cm, inner sep=0pt, draw, anchor=south west] {\Huge $11$};
\draw (8cm,-6cm) node [circle, minimum size=1.5cm, inner sep=0pt, draw, anchor=south west] {\Huge $12$};
\draw (10cm,-6cm) node [circle, minimum size=1.5cm, inner sep=0pt, draw, anchor=south west] {\Huge $13$};
\draw (12cm,-6cm) node [circle, minimum size=1.5cm, inner sep=0pt, draw=none, anchor=south west] {\Huge $14$};
\draw (14cm,-6cm) node [circle, minimum size=1.5cm, inner sep=0pt, draw, anchor=south west] {\Huge $15$};

\draw (0cm,-4cm) node [circle, minimum size=1.5cm, inner sep=0pt, draw=none, anchor=south west] {\Huge $16$};
\draw (2cm,-4cm) node [circle, minimum size=1.5cm, inner sep=0pt, draw, anchor=south west] {\Huge $17$};
\draw (4cm,-4cm) node [circle, minimum size=1.5cm, inner sep=0pt, draw=none, anchor=south west] {\Huge $18$};
\draw (6cm,-4cm) node [circle, minimum size=1.5cm, inner sep=0pt, draw, anchor=south west] {\Huge $19$};
\draw (8cm,-4cm) node [circle, minimum size=1.5cm, inner sep=0pt, draw, anchor=south west] {\Huge $20$};
\draw (10cm,-4cm) node [circle, minimum size=1.5cm, inner sep=0pt, draw=none, anchor=south west] {\Huge $21$};
\draw (12cm,-4cm) node [circle, minimum size=1.5cm, inner sep=0pt, draw, anchor=south west] {\Huge $22$};
\draw (14cm,-4cm) node [circle, minimum size=1.5cm, inner sep=0pt, draw=none, anchor=south west] {\Huge $23$};
\draw (0cm,-2cm) node [circle, minimum size=1.5cm, inner sep=0pt, draw, anchor=south west] {\Huge $24$};
\draw (2cm,-2cm) node [circle, minimum size=1.5cm, inner sep=0pt, draw=none, anchor=south west] {\Huge $25$};
\draw (4cm,-2cm) node [circle, minimum size=1.5cm, inner sep=0pt, draw, anchor=south west] {\Huge $26$};
\draw (6cm,-2cm) node [circle, minimum size=1.5cm, inner sep=0pt, draw=none, anchor=south west] {\Huge $27$};
\draw (8cm,-2cm) node [circle, minimum size=1.5cm, inner sep=0pt, draw=none, anchor=south west] {\Huge $28$};
\draw (10cm,-2cm) node [circle, minimum size=1.5cm, inner sep=0pt, draw, anchor=south west] {\Huge $29$};
\draw (12cm,-2cm) node [circle, minimum size=1.5cm, inner sep=0pt, draw=none, anchor=south west] {\Huge $30$};
\draw (14cm,-2cm) node [circle, minimum size=1.5cm, inner sep=0pt, draw, anchor=south west] {\Huge $31$};

\draw (0cm,0cm) node [circle, minimum size=1.5cm, inner sep=0pt, draw=none, anchor=south west] {\Huge $32$};
\draw (2cm,0cm) node [circle, minimum size=1.5cm, inner sep=0pt, draw, anchor=south west] {\Huge $33$};
\draw (4cm,0cm) node [circle, minimum size=1.5cm, inner sep=0pt, draw=none, anchor=south west] {\Huge $34$};
\draw (6cm,0cm) node [circle, minimum size=1.5cm, inner sep=0pt, draw=none, anchor=south west] {\Huge $35$};
\draw (8cm,0cm) node [circle, minimum size=1.5cm, inner sep=0pt, draw=none, anchor=south west] {\Huge $36$};
\draw (10cm,0cm) node [circle, minimum size=1.5cm, inner sep=0pt, draw=none, anchor=south west] {\Huge $37$};
\draw (12cm,0cm) node [circle, minimum size=1.5cm, inner sep=0pt, draw, anchor=south west] {\Huge $38$};
\draw (14cm,0cm) node [circle, minimum size=1.5cm, inner sep=0pt, draw=none, anchor=south west] {\Huge $39$};
\draw (0cm,2cm) node [circle, minimum size=1.5cm, inner sep=0pt, draw, anchor=south west] {\Huge $40$};
\draw (2cm,2cm) node [circle, minimum size=1.5cm, inner sep=0pt, draw=none, anchor=south west] {\Huge $41$};
\draw (4cm,2cm) node [circle, minimum size=1.5cm, inner sep=0pt, draw=none, anchor=south west] {\Huge $42$};
\draw (6cm,2cm) node [circle, minimum size=1.5cm, inner sep=0pt, draw=none, anchor=south west] {\Huge $43$};
\draw (8cm,2cm) node [circle, minimum size=1.5cm, inner sep=0pt, draw=none, anchor=south west] {\Huge $44$};
\draw (10cm,2cm) node [circle, minimum size=1.5cm, inner sep=0pt, draw=none, anchor=south west] {\Huge $45$};
\draw (12cm,2cm) node [circle, minimum size=1.5cm, inner sep=0pt, draw=none, anchor=south west] {\Huge $46$};
\draw (14cm,2cm) node [circle, minimum size=1.5cm, inner sep=0pt, draw, anchor=south west] {\Huge $47$};
\end{tikzpicture}}
\end{center}
\end{figure}

\vs\vs\vs\vs\vs
\begin{figure}
\caption{$s=10$}
\hspace{-1.5cm}
\begin{center}
\scalebox{0.3}{\begin{tikzpicture} 
\draw (-2cm,-8cm) node [circle, minimum size=1.5cm, inner sep=0pt, draw=none, anchor=south west] {\Huge $0$};
\draw (0cm,-8cm) node [circle, minimum size=1.5cm, inner sep=0pt, draw, anchor=south west] {\Huge $1$};
\draw (2cm,-8cm) node [circle, minimum size=1.5cm, inner sep=0pt, draw, anchor=south west] {\Huge $2$};
\draw (4cm,-8cm) node [circle, minimum size=1.5cm, inner sep=0pt, draw, anchor=south west] {\Huge $3$};
\draw (6cm,-8cm) node [circle, minimum size=1.5cm, inner sep=0pt, draw, anchor=south west] {\Huge $4$};
\draw (8cm,-8cm) node [circle, minimum size=1.5cm, inner sep=0pt, draw, anchor=south west] {\Huge $5$};
\draw (10cm,-8cm) node [circle, minimum size=1.5cm, inner sep=0pt, draw, anchor=south west] {\Huge $6$};
\draw (12cm,-8cm) node [circle, minimum size=1.5cm, inner sep=0pt, draw, anchor=south west] {\Huge $7$};
\draw (14cm,-8cm) node [circle, minimum size=1.5cm, inner sep=0pt, draw, anchor=south west] {\Huge $8$};
\draw (16cm,-8cm) node [circle, minimum size=1.5cm, inner sep=0pt, draw=none, anchor=south west] {\Huge $9$};
\draw (-2cm,-6cm) node [circle, minimum size=1.5cm, inner sep=0pt, draw, anchor=south west] {\Huge $10$};
\draw (0cm,-6cm) node [circle, minimum size=1.5cm, inner sep=0pt, draw=none, anchor=south west] {\Huge $11$};
\draw (2cm,-6cm) node [circle, minimum size=1.5cm, inner sep=0pt, draw, anchor=south west] {\Huge $12$};
\draw (4cm,-6cm) node [circle, minimum size=1.5cm, inner sep=0pt, draw, anchor=south west] {\Huge $13$};
\draw (6cm,-6cm) node [circle, minimum size=1.5cm, inner sep=0pt, draw, anchor=south west] {\Huge $14$};
\draw (8cm,-6cm) node [circle, minimum size=1.5cm, inner sep=0pt, draw, anchor=south west] {\Huge $15$};
\draw (10cm,-6cm) node [circle, minimum size=1.5cm, inner sep=0pt, draw, anchor=south west] {\Huge $16$};
\draw (12cm,-6cm) node [circle, minimum size=1.5cm, inner sep=0pt, draw, anchor=south west] {\Huge $17$};
\draw (14cm,-6cm) node [circle, minimum size=1.5cm, inner sep=0pt, draw=none, anchor=south west] {\Huge $18$};
\draw (16cm,-6cm) node [circle, minimum size=1.5cm, inner sep=0pt, draw, anchor=south west] {\Huge $19$};
\draw (-2cm,-4cm) node [circle, minimum size=1.5cm, inner sep=0pt, draw=none, anchor=south west] {\Huge $20$};
\draw (0cm,-4cm) node [circle, minimum size=1.5cm, inner sep=0pt, draw, anchor=south west] {\Huge $21$};
\draw (2cm,-4cm) node [circle, minimum size=1.5cm, inner sep=0pt, draw=none, anchor=south west] {\Huge $22$};
\draw (4cm,-4cm) node [circle, minimum size=1.5cm, inner sep=0pt, draw, anchor=south west] {\Huge $23$};
\draw (6cm,-4cm) node [circle, minimum size=1.5cm, inner sep=0pt, draw, anchor=south west] {\Huge $24$};
\draw (8cm,-4cm) node [circle, minimum size=1.5cm, inner sep=0pt, draw, anchor=south west] {\Huge $25$};
\draw (10cm,-4cm) node [circle, minimum size=1.5cm, inner sep=0pt, draw, anchor=south west] {\Huge $26$};
\draw (12cm,-4cm) node [circle, minimum size=1.5cm, inner sep=0pt, draw=none, anchor=south west] {\Huge $27$};
\draw (14cm,-4cm) node [circle, minimum size=1.5cm, inner sep=0pt, draw, anchor=south west] {\Huge $28$};
\draw (16cm,-4cm) node [circle, minimum size=1.5cm, inner sep=0pt, draw=none, anchor=south west] {\Huge $29$};
\draw (-2cm,-2cm) node [circle, minimum size=1.5cm, inner sep=0pt, draw, anchor=south west] {\Huge $30$};
\draw (0cm,-2cm) node [circle, minimum size=1.5cm, inner sep=0pt, draw=none, anchor=south west] {\Huge $31$};
\draw (2cm,-2cm) node [circle, minimum size=1.5cm, inner sep=0pt, draw, anchor=south west] {\Huge $32$};
\draw (4cm,-2cm) node [circle, minimum size=1.5cm, inner sep=0pt, draw=none, anchor=south west] {\Huge $33$};
\draw (6cm,-2cm) node [circle, minimum size=1.5cm, inner sep=0pt, draw, anchor=south west] {\Huge $34$};
\draw (8cm,-2cm) node [circle, minimum size=1.5cm, inner sep=0pt, draw, anchor=south west] {\Huge $35$};
\draw (10cm,-2cm) node [circle, minimum size=1.5cm, inner sep=0pt, draw=none, anchor=south west] {\Huge $36$};
\draw (12cm,-2cm) node [circle, minimum size=1.5cm, inner sep=0pt, draw, anchor=south west] {\Huge $37$};
\draw (14cm,-2cm) node [circle, minimum size=1.5cm, inner sep=0pt, draw=none, anchor=south west] {\Huge $38$};
\draw (16cm,-2cm) node [circle, minimum size=1.5cm, inner sep=0pt, draw, anchor=south west] {\Huge $39$};

\draw (-2cm,0cm) node [circle, minimum size=1.5cm, inner sep=0pt, draw=none, anchor=south west] {\Huge $40$};
\draw (0cm,0cm) node [circle, minimum size=1.5cm, inner sep=0pt, draw, anchor=south west] {\Huge $41$};
\draw (2cm,0cm) node [circle, minimum size=1.5cm, inner sep=0pt, draw=none, anchor=south west] {\Huge $42$};
\draw (4cm,0cm) node [circle, minimum size=1.5cm, inner sep=0pt, draw, anchor=south west] {\Huge $43$};
\draw (6cm,0cm) node [circle, minimum size=1.5cm, inner sep=0pt, draw=none, anchor=south west] {\Huge $44$};
\draw (8cm,0cm) node [circle, minimum size=1.5cm, inner sep=0pt, draw=none, anchor=south west] {\Huge $45$};
\draw (10cm,0cm) node [circle, minimum size=1.5cm, inner sep=0pt, draw, anchor=south west] {\Huge $46$};
\draw (12cm,0cm) node [circle, minimum size=1.5cm, inner sep=0pt, draw=none, anchor=south west] {\Huge $47$};
\draw (14cm,0cm) node [circle, minimum size=1.5cm, inner sep=0pt, draw, anchor=south west] {\Huge $48$};
\draw (16cm,0cm) node [circle, minimum size=1.5cm, inner sep=0pt, draw=none, anchor=south west] {\Huge $49$};
\draw (-2cm,2cm) node [circle, minimum size=1.5cm, inner sep=0pt, draw, anchor=south west] {\Huge $50$};
\draw (0cm,2cm) node [circle, minimum size=1.5cm, inner sep=0pt, draw=none, anchor=south west] {\Huge $51$};
\draw (2cm,2cm) node [circle, minimum size=1.5cm, inner sep=0pt, draw, anchor=south west] {\Huge $52$};
\draw (4cm,2cm) node [circle, minimum size=1.5cm, inner sep=0pt, draw=none, anchor=south west] {\Huge $53$};
\draw (6cm,2cm) node [circle, minimum size=1.5cm, inner sep=0pt, draw=none, anchor=south west] {\Huge $54$};
\draw (8cm,2cm) node [circle, minimum size=1.5cm, inner sep=0pt, draw=none, anchor=south west] {\Huge $55$};
\draw (10cm,2cm) node [circle, minimum size=1.5cm, inner sep=0pt, draw=none, anchor=south west] {\Huge $56$};
\draw (12cm,2cm) node [circle, minimum size=1.5cm, inner sep=0pt, draw, anchor=south west] {\Huge $57$};
\draw (14cm,2cm) node [circle, minimum size=1.5cm, inner sep=0pt, draw=none, anchor=south west] {\Huge $58$};
\draw (16cm,2cm) node [circle, minimum size=1.5cm, inner sep=0pt, draw, anchor=south west] {\Huge $59$};

\draw (-2cm,4cm) node [circle, minimum size=1.5cm, inner sep=0pt, draw=none, anchor=south west] {\Huge $60$};
\draw (0cm,4cm) node [circle, minimum size=1.5cm, inner sep=0pt, draw, anchor=south west] {\Huge $61$};
\draw (2cm,4cm) node [circle, minimum size=1.5cm, inner sep=0pt, draw=none, anchor=south west] {\Huge $62$};
\draw (4cm,4cm) node [circle, minimum size=1.5cm, inner sep=0pt, draw=none, anchor=south west] {\Huge $63$};
\draw (6cm,4cm) node [circle, minimum size=1.5cm, inner sep=0pt, draw=none, anchor=south west] {\Huge $64$};
\draw (8cm,4cm) node [circle, minimum size=1.5cm, inner sep=0pt, draw=none, anchor=south west] {\Huge $65$};
\draw (10cm,4cm) node [circle, minimum size=1.5cm, inner sep=0pt, draw=none, anchor=south west] {\Huge $66$};
\draw (12cm,4cm) node [circle, minimum size=1.5cm, inner sep=0pt, draw=none, anchor=south west] {\Huge $67$};
\draw (14cm,4cm) node [circle, minimum size=1.5cm, inner sep=0pt, draw, anchor=south west] {\Huge $68$};
\draw (16cm,4cm) node [circle, minimum size=1.5cm, inner sep=0pt, draw=none, anchor=south west] {\Huge $69$};
\draw (-2cm,6cm) node [circle, minimum size=1.5cm, inner sep=0pt, draw, anchor=south west] {\Huge $70$};
\draw (0cm,6cm) node [circle, minimum size=1.5cm, inner sep=0pt, draw=none, anchor=south west] {\Huge $71$};
\draw (2cm,6cm) node [circle, minimum size=1.5cm, inner sep=0pt, draw=none, anchor=south west] {\Huge $72$};
\draw (4cm,6cm) node [circle, minimum size=1.5cm, inner sep=0pt, draw=none, anchor=south west] {\Huge $73$};
\draw (6cm,6cm) node [circle, minimum size=1.5cm, inner sep=0pt, draw=none, anchor=south west] {\Huge $74$};
\draw (8cm,6cm) node [circle, minimum size=1.5cm, inner sep=0pt, draw=none, anchor=south west] {\Huge $75$};
\draw (10cm,6cm) node [circle, minimum size=1.5cm, inner sep=0pt, draw=none, anchor=south west] {\Huge $76$};
\draw (12cm,6cm) node [circle, minimum size=1.5cm, inner sep=0pt, draw=none, anchor=south west] {\Huge $77$};
\draw (14cm,6cm) node [circle, minimum size=1.5cm, inner sep=0pt, draw=none, anchor=south west] {\Huge $78$};
\draw (16cm,6cm) node [circle, minimum size=1.5cm, inner sep=0pt, draw, anchor=south west] {\Huge $79$};
\end{tikzpicture}}
\end{center}
\end{figure}
\clearpage
\begin{center}
\begin{figure}
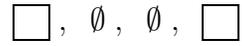


 {\bf APPENDIX B} \\
 \vspace{0.2cm}
 {\bf The $s$-quotients of ${\bf \kappa_{s\pm1}}$}

\vspace{.5cm}
\caption{4-quotient of $\kappa_{3,5}$}
\Yvcentermath1
$\yng(1)\;,\;\; \emptyset\;,\;\; \emptyset\;,\;\; \yng(1)$
\end{figure}
\end{center}
\vs\vs\vs
\begin{center}
\begin{figure}
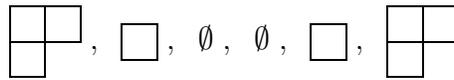

\hspace{-4cm}
\caption{6-quotient of $\kappa_{5,7}$}
 \vspace{0.1cm}
\Yvcentermath1
$\yng(2,1)\;,\;\; \yng(1)\;,\;\; \emptyset\;,\;\; \emptyset\;,\;\; \yng(1)\;,\; \;\yng(2,1)$
\end{figure}
\end{center}
\vs\vs\vs
\begin{center}
\begin{figure}
\hspace{-8cm}
\caption{8-quotient of $\kappa_{7,9}$}
\Yvcentermath1
$\yng(3,2,1)\;, \;\; \yng(2,1)\;,\;\; \yng(1)\;,\;\; \emptyset\;,\;\; \emptyset\;,\;\; \yng(1)\;,\; \;\yng(2,1)\;,\; \;\yng(3,2,1)$
\end{figure}
\end{center}
\vs\vs\vs\vs
\begin{figure}
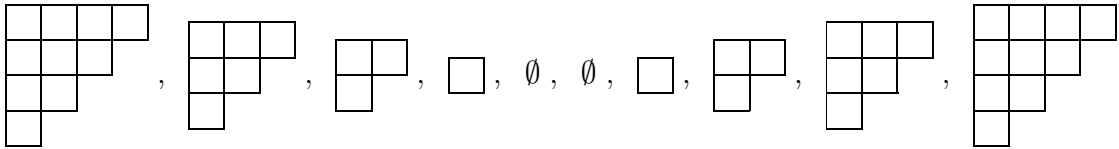


\caption{10-quotient of $\kappa_{9,11}$}
 \Yvcentermath1
\hspace{-4cm}
$$\hspace{-1.1cm}
\yng(4,3,2,1)\;,\;\;\yng(3,2,1)\;, \;\; \yng(2,1)\;,\;\; \yng(1)\;,\;\; \emptyset\;,\;\; \emptyset\;,\;\; \yng(1)\;,\; \;\yng(2,1)\;,\; \;\yng(3,2,1)\;,\;\;\yng(4,3,2,1)$$
\end{figure}
\end{center}

\end{document}